\documentclass[12pt]{amsart}
 \usepackage[dvips]{epsfig}
 \usepackage{comment}
 \usepackage{enumerate}
 \usepackage{amsgen, amstext,amsbsy,amsopn, amsthm, amsfonts,amssymb,amscd,amsmat
 h,euscript,enumerate,url,verbatim,calc,xypic, mathtools}

 \usepackage{latexsym}
 \usepackage{graphics}
 \usepackage{color}

\newcommand{\Pic}{\rm Pic\,}

\newcommand{\proset}{\,\mathrel{\lower 4pt\hbox{$\scriptscriptstyle/$}
\mkern -14mu\subseteq }\,} 
 
 \newtheorem{theorem}{Theorem}[section]
  \newtheorem{corollary}[theorem]{Corollary}
 \newtheorem{lemma}[theorem]{Lemma}
 \newtheorem{proposition}[theorem]{Proposition}

\newtheorem{remark}[theorem]{Remark}

 \newtheorem{example}[theorem]{Example}

\numberwithin{equation}{section}

\def\ker{\operatorname{ker}}
\def\coker{\operatorname{coker}}
\def\Br{\operatorname{Br}}
\usepackage{amsmath}
 \makeatother 
 
\begin{document}
\date{\today}

\title{Relative Brauer Groups and \'etale cohomology}
 \author{Vivek Sadhu} 
 
 \address{Department of Mathematics, Indian Institute of Science Education and Research, Bhopal, Bhopal Bypass Road, Bhauri, Bhopal-462066, Madhya Pradesh, India}
 \email{ vsadhu@iiserb.ac.in, viveksadhu@gmail.com}
 \thanks{Author was supported by SERB-DST MATRICS grant MTR/2018/000283}

\begin{abstract}
 In this article, we construct a natural group homomorphism 
 $$ \theta:  \Br(f)\to H^{1}_{et}(S, f_{*}\mathcal{O}_{X}^{\times}/\mathcal{O}_{S}^{\times})$$ for a faithful affine map $f: X \to S$ of noetherian schemes. Here $\Br(f)$ denotes the relative Brauer group of $f.$ We also prove $\Br(f)=0$ whenever $f: A \hookrightarrow B$ is a subintegral extension of noetherian $\mathbb{Q}$-algebras.
 Furthermore, we prove a relative version of Kummer's exact sequence. 
 \end{abstract}
\maketitle

\section{Introduction}
For a map $f: X \to S$ of schemes, we define the relative Brauer group $\Br(f)$ as the Grothendieck group of a certain relative category $Az(f^{*}),$ i.e., $\Br(f):= K_{0}(Az(f^{*}))$(see section 2). This $\Br(f)$ fits into a natural exact sequence 

$$\Pic({\it S}) \rightarrow \Pic({\it X}) \rightarrow \Br({\it f}) \rightarrow \Br({\it S}) \rightarrow \Br({\it X}).$$ Classically, the relative Brauer group is just the kernel of the natural map $\Br(S)\stackrel{f^{*}}\to \Br(X)$ and is denoted by $\Br(X/S)$. Note that the relative Brauer group $\Br(f)$ is different from the classical one $\Br(X/S)$. But in the case of field extension, i.e., if $f: L \hookrightarrow K$ is a field extension then $\Br(f)$ is isomorphic to $\Br(K|L).$ The details related to $\Br(K|L)$ can be found in  \cite{FS}, \cite{FS1}, \cite{FS2}. In general, we have a natural map $\Br(f) \to \Br(X/S)$ and, this map is an isomorphism if and only if $\Pic({\it S}) \to \Pic({\it X})$ is surjective. . 

Now assume that $f: X \to S$ is a faithful affine map of schemes, i.e., affine and the structure map $\mathcal{O}_{S}\rightarrow f_{*}\mathcal{O}_{X}$ is injective. We define 
$$ \Br^{'}(f):= H_{et}^{1}(S,   f_{*}\mathcal{O}_{X}^{\times}/\mathcal{O}_{S}^{\times}).$$ If $f: L \hookrightarrow K$ is a finite field extension then $\Br(f)\cong \Br^{'}(f)$ (see Lemma \ref{rel}). But $\Br(f)$ and $\Br^{'}(f)$ are different in general (see Example \ref{not iso in gen}). This is why we prefer to call the latter group as a relative cohomological Brauer group. One of the goals of this article is to relate $\Br(f)$ and $\Br^{'}(f)$. We prove the following (see Theorem \ref{natural map}):

\begin{theorem}\label{main th}
 Let $f: X\to S $ be a faithful affine map of noetherian schemes. Then there is a natural group homomorphism $\theta: \Br(f)\to \Br^{'}(f):= 
 H^{1}_{et}(S, f_{*}\mathcal{O}_{X}^{\times}/ \mathcal{O}_{S}^{\times}).$ \end{theorem}
We give an outline of the proof of Theorem \ref{main th}. It is known that there is a natural isomorphism $\check H_{et}^{1}(S, \mathcal{F})\stackrel{\cong}\to  H_{et}^{1}(S, \mathcal{F})$ for every abelian sheaf $\mathcal{F}.$ Here $\check H_{et}^{1}(S, \mathcal{F})$ denotes the first \'etale \v{C}ech cohomology group associated to $\mathcal{F}.$ So we can restrict the target of the map $\theta$ to the first \'etale \v{C}ech cohomology group associated to $ (f_{*}\mathcal{O}_{X}^{\times}/\mathcal{O}_{S}^{\times})_{et}.$  Let  $\mathcal{P}ic^{{\it f}}_{et}$ be the \'etale sheaf associated to the presheaf $\Pic^{{\it f}}$ on $S_{et},$ defined as $\Pic^{{\it f}} ({\it U}) = \Pic({\it f_{U})},$ where ${\it f}_{U}: X\times_S U \to U$ and the group $\Pic({\it f_{U}})$ is generated by pairs $(L_{1}, L_{2})$ of line bundles on $U$ together with an isomorphism $f_{U}^{*}L_{1}\cong f_{U}^{*}L_{2}$ of $\mathcal{O}_{X\times_S U}$-modules with suitable relations (for details, see \cite{SW1}). There is a natural isomorphism $(f_{*}\mathcal{O}_{X}^{\times}/\mathcal{O}_{S}^{\times})_{et}\cong \mathcal{P}ic^{f}_{et}$ (see Lemma \ref{same as etale}). Some well-known  facts pertaining to (pre)-sheaf torsors and the first \'etale \v{C}ech cohomology group allow us to restrict further the target of the map $\theta$ to the set of isomorphism classes of $\Pic^{{\it f}}$-torsors. Therefore, the problem boils down to defining a $\Pic^{{\it f}}$-torsor associated to each element of $\Br(f).$ Given an Azumaya algebra $A$ over a scheme $X,$ one can associate a fibered category $F_{A}$  over $X_{et}$ which in fact defines an element in  $H_{et}^{2}(X, \mathcal{O}_{X}^{\times}).$ More explicitly, $A\mapsto F_{A}$ defines a natural map $\Br(X) \to H_{et}^{2}(X, \mathcal{O}_{X}^{\times})$ (see  \cite{Gir}, \cite{Milne}). We follow a similar approach to prove our desired assertion. Note that the relative $\Br(f)$ is the abelian group generated by isomorphism classes of objects in $Az(f^{*})$ modulo certain relations. An object of $Az(f^{*})$ is a triple $(A_{1}, \alpha, A_{2}),$ where $A_{1}, A_{2}$ are Azumaya algebras over $S$ and $\alpha: f^{*}A_{1}\cong f^{*}A_{2}$ is an isomorphism in a suitable category. So we have the fibered categories $F_{A_{1}}$ and $F_{A_{2}}$ associated to $A_{1}$ and $A_{2}.$ By using the categories $F_{A_{1}}$ and $F_{A_{2}},$ we define a relative category $G_{A}$ over $S_{et}$  associated to $A:= (A_{1}, \alpha, A_{2}).$ Further, using the category $G_{A}$, we construct a $\Pic^{{\it f}}$-torsor associated to $A:= (A_{1}, \alpha, A_{2})$ (see (\ref{first step}) and (\ref{second step})). A significant portion of section 4 is dedicated to the construction of such a presheaf torsor and proving its required properties.

Next, we discuss the relative Brauer group of subintegral extensions. We say that an extension $A\hookrightarrow B$ is {\it subintegral} if $B$ is integral over $A$ and ${\rm{{\rm Spec}}}(B) \to {\rm{Spec}}(A)$ is a bijection inducing isomorphisms on all residue fields. For example, $\mathbb{C}[t^{2}, t^{3}]\hookrightarrow \mathbb{C}[t]$ is a subintegral extension. We show that if $f: A \hookrightarrow B$ is a subintegral extension of noetherian $\mathbb{Q}$-algebras then the induced map $\Br(A) \stackrel{f^{*}}\to \Br(B)$ is an isomorphism and $\Br(f)=0$ (see Theorem \ref{vanishing}).

We further study the Kummer exact sequence in the relative setting.  Write $\mathcal{I}_{et}$ for the \'etale sheaf  $f_{*}\mathcal{O}_{X}^{\times}/\mathcal{O}_{S}^{\times}.$ Let $\mathcal{\mu}_{n}^{f}$ denotes the kernel of $\mathcal{I}_{et}\stackrel{n} \to \mathcal{I}_{et}.$ We prove that if $f: X\to S$ is a faithful finite map of schemes (i.e. finite and the structure map $\mathcal{O}_{S}\rightarrow f_{*}\mathcal{O}_{X}$ is injective) and characteristic of $k(s)$ does not divide $n$ for any $s\in S$ then the sequence  
\begin{align*}
   0 \to \mu_{n}^{f} \to \mathcal{I}_{et}\stackrel{n}\to \mathcal{I}_{et} \to 0 
 \end{align*}
 of \'etale sheaves is exact (see Proposition \ref{Kummer sequence}). As an application, we obtain  $H_{et}^{0}(S, \mu_{n}^{f})\cong {_n\Pic({\it f})}$ and the following short exact sequence (see Theorem \ref{cohomology of roots unity})
 \begin{align*}
    0\to \Pic({\it f}) \otimes \mathbb{Z}/n\mathbb{Z} \to {\it H_{et}^{1}(S, \mu_{n}^{f}) \to  {_n{\it H}_{et}^{1}(S, \mathcal{I}_{et})}} \to 0,
  \end{align*} where $\Pic({\it f})$ is the relative Picard group of $f$ (see \cite{SW1}). We also show that if $f: A\hookrightarrow B$ is a finite subintegral extension of noetherian $\mathbb{Q}$-algebras then $H_{et}^{i}({\rm{Spec}}(A), \mu_{n}^{f})=0$ for $i\geq 0$ (see Theorem \ref{cohomology for subintegral map}).
  
  \medskip
  
  {\bf Acknowledgements:} The author is grateful to Charles Weibel for his helpful comments during the preparation of this article. He would also like to thank the referee for valuable comments and suggestions.
\section{Relative Brauer groups}\label{basic}
In this section, we define the notion of relative Brauer groups. The most of the material in this section was developed in \cite{BA}, \cite{AG}, \cite{ Bass-Tata}, \cite{b},  \cite{VNR} by various authors.

Recall that two Azumaya algebras $A$ and $A^{'}$ over a scheme $X$ are said to be {\bf similar} if there
 exist locally free $\mathcal{O}_{X}$-modules $E$ and $E^{'},$ of finite rank over $\mathcal{O}_{X},$ such that 
 $$A \otimes_{\mathcal{O}_{X}} End(E) \cong A^{'}\otimes_{\mathcal{O}_{X}} End(E^{'}).$$ The set of similar classes of Azumaya algebras on $X$ forms an abelian group under $\otimes_{\mathcal{O}_{X}},$ which is known as the Brauer group $\Br(X)$ of $X.$ If $[A]\in \Br(X)$ then $[A]^{-1}=[A^{op}],$ where $A^{op}$ denotes the opposite algebra of $A.$ The cohomological Brauer group of $X$ is $H_{et}^{2}(X, \mathcal{O}_{X}^{\times})$ and is denoted by $\Br^{'}(X).$ The Picard group of $X$ is denoted by $\Pic({\it X})$. 
 
 Given two Azumaya algebras $A$ and $B$ over $X,$ let 
 $\Delta(A, B)$ (resp. $\tilde \Delta(A, B)$)  be the set consisting of all triples $(P, u, Q)$ with  $P$, $Q$ are (resp. self dual) locally free $\mathcal{O}_{X}$-modules of finite rank and $$u: A\otimes End(P)\to B\otimes End(Q)$$ is an isomorphism of algebras.  We define an equivalence relation $\sim $ on $\Delta(A, B)$ by   $(P, u, Q) \sim (P^{'}, u^{'}, Q^{'})$ if and only if there exist locally free $\mathcal{O}_{X}$-modules $E$ and $E^{'}$ of finite rank over $X$ and $P\otimes E\cong P^{'} \otimes E^{'},$  $Q\otimes E\cong Q^{'} \otimes E^{'}.$ 
 By considering $E$ and $E^{'}$ are to be self dual locally free $\mathcal{O}_{X}$-modules, one can check that $\sim$ is an equivalence relation on $\tilde \Delta(A, B)$ as well. 
 
 \begin{remark}\label{equi}{\rm
  Note that  $u$ and $u^{'}$ do not play any role for the equivalence relation $\sim $ on $\Delta(A, B)$ (resp. $\tilde \Delta(A, B)).$ In fact $(P, u, Q)$ and $(P, v, Q)$ both lie in a same class whenever $u, v$ are possibly distinct $\mathcal{O}_{X}$-algebra isomorphisms between $A\otimes End(P)$ and $B\otimes End(Q).$}
 \end{remark}

  \subsection*{The category $Az(X)$} $Az(X)$ is the category whose objects are the Azumaya algebras over a scheme $X$ and the set of morphisms between two objects $A$ and $B$ is defined by $$Hom_{Az(X)}(A, B):= \Delta(A, B)/\sim.$$
  Let $\phi_{1}=[(P_{1}, u_{1}, Q_{1})]\in Hom_{Az(X)}(A, B),$ $\phi_{2}= [(P_{2}, u_{2}, Q_{2})] \in  Hom_{Az(X)}(B, C).$ Then we define the composition $\phi_{2}\phi_{1}:= [(P_{1}\otimes P_{2}, u, Q_{1}\otimes Q_{2})]$ with $u: A\otimes End(P_{1}\otimes P_{2})\cong A\otimes End(P_{1})\otimes End(P_{2})\stackrel{u_{1}}\cong B\otimes End(Q_{1})\otimes End(P_{2})\stackrel{u_{2}}\cong  C\otimes End(Q_{2})\otimes End(Q_{1})\cong C\otimes End(Q_{1}\otimes Q_{2}).$ One can easily check that this composition is independent of the chosen representatives of $\phi_{1}$ and $\phi_{2}.$ If $\phi= [(P, u, Q)]\in Hom_{Az(X)}(A, B)$ then $\phi$ is an isomorphism with inverse $\phi^{-1}= [(Q, u^{-1}, P)]$ (see Remark 5.7 of \cite{VNR}). The tensor product $\otimes$ on $\mathcal{O}_{X}$-algebras defines a product on $Az(X).$

\subsection*{The category $\Omega Az(X)$} $\Omega Az(X)$ is the category consisting of all couples $(A, \phi)$ with $A \in Az(X)$ and $\phi$ is an automorphism in $Az(X).$ A morphism $(A, \phi) \to (B, \varphi)$ is a morphism $h: A\to B$ in $Az(X)$ such that $\varphi h= h \phi.$ Note that $\Omega Az(X)$ is a category with product and composition in the sense of \cite{Bass-Tata}.  We refer to 5.12 of \cite{VNR} for product and composition rules in $\Omega Az(X).$
 
 We now state a result from \cite{VNR}.
 
 \begin{lemma}\label{known iso}
  Let $X$ be a scheme. Then there are natural isomorphisms of abelian groups
  \begin{enumerate}
   \item $K_{0}(Az(X))\stackrel{\cong}\to \Br(X);$
   \item  $K_{0}(\Omega Az(X)) \stackrel{\cong} \to \Pic({\it X}).$
  \end{enumerate}
  
  \begin{proof}
   We get the assertions by applying Theorem 5.9 and Theorem 5.13 of \cite{VNR} to the category of $\mathcal{O}_{X}$-modules.
  \end{proof} \end{lemma}

 For a map $f: X\to S$ of schemes, we have the base change functor $f^{*}:Az(S)\to Az(X)$.
 
\subsection*{The category $Az(f^{*})$} $Az(f^{*})$ is the category consisting of all triples $(A, \alpha, B)$ where $A, B \in Az(S)$ and $\alpha: f^{*}(A)\to f^{*}(B)$ is an isomorphism in $Az(X).$ Here $\alpha= [(P, u, Q)]$ with $P$, $Q$ are locally free $\mathcal{O}_{X}$-modules of finite rank and $u: f^{*}A \otimes End(P)\cong f^{*}B \otimes End(Q)$(i.e. $f^{*}A$ is similar to $f^{*}B$). A morphism $(A, \alpha, B) \to (A^{'}, \alpha^{'}, B^{'})$ is a pair $(u, v),$ where $u: A\to A^{'}$ and $v: B\to B^{'}$ are morphisms in $Az(S)$ such that $\alpha^{'} f^{*}u= f^{*}v~ \alpha$.  We refer to 5.17 of \cite{VNR} for product and composition rules in $Az(f^{*}).$
 
 Following Bass \cite{Bass-Tata}, we define the {\bf relative Brauer group} $\Br(f)$ of $f$ to be the Grothendieck group $K_{0}(Az(f^{*})),$ i.e., the abelian group generated by $[(A_{1}, \alpha, A_{2})],$ where $ (A_{1}, \alpha, A_{2})\in Az(f^{*}),$ and with the following relations:
\begin{enumerate}
 \item $[(A_{1}, \alpha, A_{2})]+ [(A_{1}^{'}, \alpha^{'}, A_{2}^{'})]= [(A_{1}\otimes A_{1}^{'}, \alpha\otimes \alpha^{'}, A_{2}\otimes A_{2}^{'})]; $
 \item $[(A_{1}, \alpha, A_{2})] + [(A_{2}, \beta, A_{3})]= [(A_{1}, \beta\alpha, A_{3})].$
\end{enumerate}

\begin{remark}\label{presentation}\rm{
 By (2), $[(A, 1, A)]=0$ for any $A\in Az(S)$. Then by using (1), we get that every element of $\Br(f)$ has the form $[(A, \alpha, End_{\mathcal{O}_{S}}(B))],$ where $A$ and $B$ are in $Az(S)$. For general details, see Remark 1.3 in Chapter 1 of \cite{Bass-Tata}.}
\end{remark}

By Theorem 5.18 of \cite{VNR} (more precisely, Example 5.19(2) of \cite{VNR}), there is a natural exact sequence of abelian groups for each $f: X\to S$, 
\begin{align}\label{Pic-Br}
 \Pic({\it S}) \stackrel{{\it f}^{*}}\to \Pic({\it X}) \stackrel{ \partial}\to \Br({\it f}) \stackrel{\varrho}\to \Br({\it S}) \stackrel{{\it f}^{*}}\to \Br({\it X}).
\end{align}
Here $\varrho([(A_{1}, \alpha, A_{2})])= A_{1}\otimes A_{2}^{op}$ and the map $\partial$ is defined as follows. By Lemma \ref{known iso}(2), we can identify $\Pic({\it X})$ with $K_{0}(\Omega Az(X)).$ Let $Az^{'}(X)$ denote the full subcategory of $Az(X)$ whose objects are all $f^{*}A,$ $A\in Az(S).$  The functor $f^{*}: Az(S) \to Az(X)$ is cofinal (see Proposition 5.15 of \cite{VNR}). So, $Az^{'}(X)$ is a cofinal subcategory of $Az(X).$ We refer to p. 19 of \cite{Bass-Tata} for general details related to cofinal functors and cofinal subcategories. Now the Theorem 3.1(b) of \cite{Bass-Tata} says that there is an isomorphism $\kappa: K_{0}(\Omega Az^{'}(X))\to K_{0}(\Omega Az(X)).$ We also have a natural homomorphism $\partial_{1}: K_{0}(\Omega Az^{'}(X)) \to K_{0}(Az(f^{*})),$ sending $[(f^{*}A, \alpha)]$ to $[(A, \alpha, A)].$  Hence we define $\partial: K_{0}(\Omega Az(X)) \to K_{0}(Az(f^{*}))$ as $\partial_{1} \kappa^{-1}.$

\begin{remark}\label{Br=Pic}\rm{
 Suppose that $S= \rm{Spec}(\mathbb{Z}),$ i.e., $f: X \to \rm{Spec}(\mathbb{Z}).$ By Proposition 4.2 of \cite{MS}, $\Br(\mathbb{Z})=0.$ Then the sequence (\ref{Pic-Br}) implies that $\Br(f)\cong \Pic({ \it X}).$ }
\end{remark}
Next, we define a slight variant of the group $\Br(f)$ which will be used in the section \ref{main}.
\subsection*{The categories $\tilde Az(X)$ and $\tilde Az(f^{*})$} $\tilde Az(X)$ is the category whose objects are the Azumaya algebras over a scheme $X$ and the set of morphisms between two objects $A$ and $B$ is defined by $$Hom_{Az(X)}(A, B):= \tilde \Delta(A, B)/\sim.$$ $\tilde Az(f^{*})$ is the category consisting of all triples $(A, \alpha, B)$ where $A, B \in \tilde Az(S)$ and $\alpha: f^{*}(A)\to f^{*}(B)$ is an isomorphism in $\tilde Az(X).$ A morphism $(A, \alpha, B) \to (A^{'}, \alpha^{'}, B^{'})$ is a pair $(u, v),$ where $u: A\to A^{'}$ and $v: B\to B^{'}$ are morphisms in $\tilde Az(S)$ such that $\alpha^{'} f^{*}u= f^{*}v~ \alpha$. One can define a product and composition on $\tilde Az(X)$ (resp. $\tilde Az(f^{*})$) in a similar way as it is defined on $ Az(X)$ (resp. $Az(f^{*})$). 

We define $\tilde \Br(f):= K_{0}(\tilde Az(f^{*})).$ Given $A:= (A_{1}, \alpha, A_{2}) \in Az(f^{*})$ with $\alpha = [(P, u, Q)],$ let $\mathbb{E}nd (A):= (End_{\mathcal{O}_{S}}(A_{1}), End(\alpha), End_{\mathcal{O}_{S}}(A_{2})),$ where $$End(\alpha):= [(End_{\mathcal{O}_{X}}(P), End(u), End_{\mathcal{O}_{X}}(Q))]$$ with
\begin{align*}
 End(u):  f^{*}End_{\mathcal{O}_{S}}(A_{1}) \otimes End(End(P))&\cong End_{\mathcal{O}_{X}}(f^{*}A_{1}\otimes End(P))\\ &\cong End_{\mathcal{O}_{X}}(f^{*}A_{2}\otimes End(Q)) ~({\rm using}~ u)\\ &\cong f^{*}End_{\mathcal{O}_{S}}(A_{2}) \otimes End(End(Q)).
\end{align*}
Observe that $(End_{\mathcal{O}_{X}}(P), End(u), End_{\mathcal{O}_{X}}(Q))\in \tilde \Delta(  f^{*}End_{\mathcal{O}_{S}}(A_{1}),  f^{*}End_{\mathcal{O}_{S}}(A_{2})).$ So $\mathbb{E}nd (A) \in \tilde Az(f^{*}).$
We define a map $\mathbb{E}nd:  ob( Az(f^{*})) \to \tilde \Br(f)$ by $A\mapsto [\mathbb{E}nd (A)].$ The following facts are easy to verify:
\begin{enumerate}
 \item If $A:= (A_{1}, \alpha, A_{2})\cong B:= (B_{1}, \beta, B_{2})$ in $Az(f^{*})$ then $[\mathbb{E}nd (A)]\cong [\mathbb{E}nd (B)]$ in $\tilde \Br(f).$
 \item  
 $ [\mathbb{E}nd (A_{1}\otimes A_{1}^{'}, \alpha \otimes \alpha^{'}, A_{2}\otimes A_{2}^{'})]= [\mathbb{E}nd(A_{1}, \alpha, A_{2})]+ [\mathbb{E}nd(A_{1}^{'}, \alpha^{'}, A_{2}^{'})],$ for every pair of objects $(A_{1}, \alpha, A_{2}), (A_{1}^{'}, \alpha^{'}, A_{2}^{'})$ in $Az(f^{*}).$
 \item   $[\mathbb{E}nd(A_{1}, \beta \alpha, A_{3})]= [\mathbb{E}nd(A_{1},  \alpha, A_{2})] + [\mathbb{E}nd(A_{2}, \beta, A_{3})],$ for every pair of objects $(A_{1},  \alpha, A_{2}), (A_{2}, \beta, A_{3})$ in $Az(f^{*}).$
\end{enumerate}
Therefore, the map $\mathbb{E}nd$ induces a natural group homomorphism 
\begin{align}\label{Br to tilde Br}
 \Upsilon: \Br(f) \to \tilde \Br(f)
\end{align}
sending $[(A_{1}, \alpha, A_{2})]$ to $[\mathbb{E}nd((A_{1}, \alpha, A_{2}))].$

\section{Relative Cohomological Brauer Groups}\label{introduce coh br}

 Let $f:L\hookrightarrow K$ be a field extension. Then the relative Brauer group in the classical sense is just the kernel of the natural map $\Br(L)\rightarrow \Br(K)$ and it was denoted by  $\Br(K|L)$(for details, see \cite{FS}, \cite{FS1}, \cite{FS2}). We begin with the following observation.
 
 \begin{lemma}\label{rel}
  If $f: L \hookrightarrow K$ is a finite field extension then 
$H_{et}^{1}({\rm Spec}(L), f_{*}\mathcal{O}^{\times}_K/\mathcal{O}^{\times}_{L})\cong\Br(K|L)$.
 \end{lemma}
 
 \begin{proof}
  For a field $F$, $H^{2}_{et}(F, \mathcal{O}^{\times}_{F})\cong Br(F)$ and $H^{1}_{et}(F, \mathcal{O}^{\times}_{F})\cong \Pic({\it F})=0.$ Then (by using the exact sequence (\ref{Pic-Br}))\begin{align*}
         H_{et}^{1}({\rm Spec}(L), f_{*}\mathcal{O}^{\times}_K/\mathcal{O}^{\times}_{L}) &\cong\ker [H^{2}_{et}(L, \mathcal{O}^{\times}_{L}) \to  H^{2}_{et}(K, \mathcal{O}^{\times}_{K})] \\
         &\cong \ker [\Br(L)\to \Br(K)]\\
         &\cong \Br(K|L) 
         \end{align*}\end{proof}

 Motivated by the above lemma,  we define   $$ \Br^{'}(f):= H_{et}^{1}(S,   f_{*}\mathcal{O}_{X}^{\times}/\mathcal{O}_{S}^{\times})$$ for a faithful affine map $f: X \rightarrow S$ of schemes. We call it {\bf relative cohomological Brauer group} of $f$.
 
 \begin{example}\label{iso in hensel}\rm{
   Suppose that $f: X \to S$ is faithful finite, i.e., finite and the structure map $\mathcal{O}_{S}\rightarrow f_{*}\mathcal{O}_{X}$ is injective and $S={\rm{{\rm Spec}}}(A)$, where $A$ is a hensel local ring. Then $X= {\rm{{\rm Spec}}}(B)$ and $B$ is a finite product of hensel local rings. In this situation, $\Br(S)\cong H_{et}^{2}(S, \mathcal{O}_{S}^{\times})$ and $\Br(X) \hookrightarrow H_{et}^{2}(X, \mathcal{O}_{X}^{\times})$ by Corollary 2.12 of \cite{Milne}. Therefore, by a similar argument as Lemma \ref{rel}, $\Br(f)\cong \Br^{'}(f).$}
 \end{example}
 
 \begin{example}\label{not iso in gen}\rm{
 In general, $\Br(f)$ and $\Br^{'}(f)$ are not isomorphic for any faithful affine $f: X\to S.$ For example, let $S={\rm{{\rm Spec}}}(A)$, where $A$ is a strictly hensel local ring. Then $X= {\rm{{\rm Spec}}}(B),$ and $\Br^{'}(f)=0$ because the higher cohomology vanishes for a strictly henselian ring. We also have  $\Br(S)=0.$ Therefore $\Br(f)\cong \Pic({\it B})$ by (\ref{Pic-Br}).} In particular, one can consider $\mathbb{C} \subset \mathbb{C}[t^{2}, t^{3}].$
 \end{example}
 
 \begin{example}\label{torsion}\rm{
  Suppose that $f: X \to S$ is faithful finite and $S={\rm{{\rm Spec}}}(A)$, where $A$ is a local ring. Then $X= {\rm{{\rm Spec}}}(B)$ and $B$ is a semilocal ring because $B$ is finite over a local ring. By (\ref{Pic-Br}),
  $$ 0\to \Br(f) \to \Br(A) \to \Br(B)$$ is exact and $\Br(f)\hookrightarrow \Br^{'}(f).$  We know that if a scheme $X$ has only finitely many connected components then $\Br(X)$ is torsion (see Proposition 2.7 of \cite{Milne}). Therefore $\Br(f)$ is torsion. By Theorem 1.1 of \cite{Djong}, $\Br(A)\cong H_{et}^{2}(A, \mathcal{O}_{A}^{\times})_{tor}$ and same holds for $B$. In fact, we get $\Br(f)\cong H_{et}^{1}(S, f_{*}\mathcal{O}_{X}^{\times}/\mathcal{O}_{S}^{\times})_{tor}.$}
 \end{example}

 \begin{example}\label{not always torsion}\rm{
  The group $\Br(f)$ is not always torsion. For example, consider $f: X={\rm {\rm Spec}}(\mathbb{Q}[t^2-t, t^3-t^2]) \to S={\rm {\rm Spec}} (\mathbb{Q}).$ Then $\Pic({\it X})\cong \mathbb{Q}^{\times}.$ By (\ref{Pic-Br}), we get that $\Br(f)$ contains the torsion free subgroup $ \mathbb{Q}^{\times}.$} 
 \end{example}

 \section{Main Theorem}\label{main}
 The main goal of this section is to construct a natural group homomorphism $\Br(f) \to H^{1}_{et}(S, f_{*}\mathcal{O}_{X}^{\times}/\mathcal{O}_{S}^{\times}).$  Throughout this section, $f: X \to S$ assumed to be a faithful affine map of schemes and  $f_{U}$ always denotes the map $ X\times_{S} U \to U.$ 
 
 \subsection*{The category $F_{A}$}Given an Azumaya algebra $A$ on $X$ one can associate a category $F_{A}$ over $X_{et}$ as follows (see p. 145 of \cite{Milne}).   For an \'etale map $j: U \to X,$ an object of $F_{A}(U)$ is a pair $(E, \tau),$ where $E$ is a locally free $\mathcal{O}_{U}$-module of finite rank and $\tau$ is an isomorphism $End(E)\cong j^{*}A;$ a morphism $(E, \tau) \to (E^{'}, \tau^{'})$ is an isomorphism $E \to E^{'}$  such that the obvious diagram 
 
   \begin{equation}\label{commutativity}
     \xymatrix{
 End(E^{'}) \ar[rd]^{\tau^{'}} \ar[r] & End(E) \ar[d]^{\tau} \\
 &  j^{*}A}
   \end{equation}

   commutes.

  {\it Notation}: Given a sheaf $\mathcal{F}$ of $\mathcal{O}_{X}$-modules, we write $\mathcal{F}^{\vee}$ for the dual of $\mathcal{F}.$ 
   \subsection*{The category $G_{A}$} Given an object $A:= (A_{1}, \alpha, A_{2})$ in $\tilde Az(f^{*})$ (or in $Az(f^{*})$), we can associate a category $G_{A}$ over $S_{et}$ as follows. For an \'etale map $j: U \to S,$ an object of $G_{A}(U)$ is a triple $((E_{1}, \tau_{1}), \nu, (E_{2}, \tau_{2})),$ where $(E_{i}, \tau_{i}) \in F_{A_{i}}(U)$ for $i=1,2$ and $\nu$ is an isomorphism ${f_{U}^{*}}E_{1} \otimes ({f_{U}^{*}}E_{2})^{\vee}\cong  ({f_{U}^{*}}E_{1})^{\vee} \otimes {f_{U}^{*}}E_{2}$ of $\mathcal{O}_{X\times_{S} U}$-modules.  Here $F_{A_{i}}$ are the categories associated to $A_{i}$.  Similarly, a morphism $((E_{1}, \tau_{1}), \nu, (E_{2}, \tau_{2}) ) \to ((E_{1}^{'}, \tau_{1}^{'}), \nu^{'}, (E_{2}^{'}, \tau_{2}^{'}))$ is a pair of isomorphisms $E_{1} \to E_{1}^{'}$ and $E_{2} \to E_{2}^{'}$ such that the diagram like (\ref{commutativity}) and 
    \begin{equation}\label{comm2}
     \xymatrix{
   {f_{U}^{*}}E_{1} \otimes ({f_{U}^{*}}E_{2})^{\vee} \ar[d]^{\nu} \ar[r] &  {f_{U}^{*}}E_{1}^{'} \otimes({f_{U}^{*}}E_{2}^{'})^{\vee} \ar[d]^{\nu^{'}} \\
  ( {f_{U}^{*}}E_{1})^{\vee} \otimes{f_{U}^{*}}E_{2} \ar[r] &  {(f_{U}^{*}}E_{1}^{'})^ {\vee} \otimes {f_{U}^{*}}E_{2}^{'}}
   \end{equation} commute.
 
 Let $[G_{A}(U)]$ be the set of all isomorphism classes of objects of $G_{A}(U).$ We write $[((E_{1}, \tau_{1}), \nu, (E_{2}, \tau_{2}))]$ for the isomorphism class of $((E_{1}, \tau_{1}), \nu, (E_{2}, \tau_{2})).$ Then 
 the assignment $$U \mapsto [G_{A}(U)]$$ is a presheaf of sets on $S_{et},$ where the restriction maps are given by pullbacks. It is denoted by $[G_{A}].$

\subsection*{Presheaf torsors} Let $\mathcal{C}$ be a site. Let $\mathcal{G}$ be a presheaf of groups on $\mathcal{C}.$ A $\mathcal{G}$-torsor is a presheaf of sets $\mathcal{F}$ on $\mathcal{C}$ equipped with an action $\rho: \mathcal{G} \times \mathcal{F} \to \mathcal{F}$ such that
  \begin{enumerate}
   \item  the action $\rho(U): \mathcal{G}(U) \times \mathcal{F}(U) \to \mathcal{F}(U)$ is simply transitive provided $\mathcal{F}(U)$ is nonempty.
   \item for every $U\in ob(\mathcal{C})$ there exists a covering $\{U_{i} \to U\}_{i\in I}$ of $U$ such that $\mathcal{F}(U_{i})$ is nonempty for all $i$.
   \end{enumerate}

   A morphism of $\mathcal{G}$-torsors $\mathcal{F} \to \mathcal{F}^{'}$ is just a morphism of presheaves of sets compatible with the $\mathcal{G}$-action. A trivial $\mathcal{G}$-torsor is the presheaf $\mathcal{G}$ with obvious left $\mathcal{G}$-action. Note that a morphism between $\mathcal{G}$-torsors is always an isomorphism. Moreover, a $\mathcal{G}$-torsor $\mathcal{F}$ is trivial if and only if $\Gamma(\mathcal{C}, \mathcal{F})\neq \emptyset$ (see chapter 21 of \cite{sp}). 

 \subsection*{Relative Picard groups}The relative $\Pic ({\it f})$ is the abelian group  generated by \newline $[L_1,\alpha,L_2]$, where the $L_i$ are 
line bundles on $S$ and $\alpha:f^*L_1\to f^*L_2$ is an isomorphism.
The relations are:
\begin{enumerate}
\item $[L_1,\alpha,L_2] + [L'_1,\alpha',L'_2] = 
[L_1\otimes L'_1,\alpha\otimes\alpha',L_2\otimes L'_2]$;
\item $[L_1,\alpha,L_2] + [L_2,\beta,L_3] = [L_1,\beta\alpha,L_3]$;
\item $[L_1,\alpha,L_2]=0$ if $\alpha=f^*(\alpha_0)$ 
for some $\alpha_0:L_1\cong L_2$.
\end{enumerate}

This relative Picard group $\Pic ({\it f})$ fits into the following exact sequence 

\begin{align}\label{U-Pic}
 1 \to \mathcal{O}^{\times}(S)\to \mathcal{O}^{\times}(X) \to \Pic({\it f}) \to \Pic({ \it S}) \to \Pic({ \it X}).
\end{align} Some relevant details and basic properties can be found in \cite{SW}, \cite{SW1}.

Next, we define a group which is a slight variant of the group $\Pic ({\it f}).$

\subsection*{The group $\tilde \Pic(f)$} The relative $\tilde \Pic(f)$ is the abelian group generated by $[L_{1}, \alpha, L_{2}],$ where the $L_{i}$ are line bundles on $S$ and $\alpha: f^{*}L_{1}\otimes f^{*}L_{2}^{-1}\to f^{*}L_{1}^{-1}\otimes f^{*}L_{2}$ is an isomorphism. The relations are similar to $\Pic({\it f}).$
 
 There is a natural group homomorphism 
 \begin{align}\label{tor map}
  \Psi: \tilde \Pic(f) \to \Pic({\it f})
 \end{align}
sending $[L_{1}, \alpha, L_{2}]$ to $[L_{1}\otimes L_{2}^{-1}, \alpha, L_{1}^{-1}\otimes L_{2}].$
 
 Let $\Pic^{{\it f}}$ be the \'etale presheaf on $S_{et},$ defined as  $\Pic^{{\it f}}({\it U})=  \Pic({{\it f_{U}}}).$ Similarly, we can define the \'etale presheaf $\tilde \Pic^{f}$ on $S_{et}.$

\begin{lemma}\label{iso lemma}
 Let $V_{1}, V_{2}$ and $V_{3}$ be three locally free $\mathcal{O}_{X}$-modules of finite rank. Suppose that $\iota: V_{1} \otimes V_{3} \to V_{2} \otimes V_{3}$ is an isomorphism as an $\mathcal{O}_{X}$-module. Then there is an $\mathcal{O}_{X}$-module isomorphism $V_{1} \to V_{2}.$
\end{lemma}

\begin{proof}
 We know that $\mathcal{F}\otimes_{End(\mathcal{F})} \mathcal{F}^{\vee}\cong \mathcal{O}_{X}$ for any locally free $\mathcal{O}_{X}$-module $\mathcal{F}$ of finite rank. Then we get 
 $V_{1} \stackrel{\cong}\to  V_{1}\otimes V_{3}\otimes_{End(V_{3})} V_{3}^{\vee} \stackrel{\iota\otimes id}\to V_{2}\otimes V_{3}\otimes_{End(V_{3})} V_{3}^{\vee}\stackrel{\cong}\to V_{2}.$\end{proof}

\begin{lemma}\label{line element}
Let $A:=(A_{1}, \alpha, A_{2})$ be in $\tilde Az(f^{*})$(or in $ Az(f^{*})$). Let $((E_{1}, \tau_{1}), \nu, (E_{2}, \tau_{2}) ),$ $((E_{1}^{'}, \tau_{1}^{'}), \nu^{'}, (E_{2}^{'}, \tau_{2}^{'}))$ be any two objects  in $G_{A}(U),$ where $U\in ob( S_{et}).$ Then
\begin{enumerate}
 \item we can construct a unique element of $\tilde \Pic^{f}(U).$ Call it $[L_{1}, \beta, L_{2}].$
 \item $((L_{1}\otimes E_{1}, \tau_1 a_1), \beta\otimes \nu, (L_{2}\otimes E_{2}, \tau_2 a_2))$ defines an object in $G_{A}(U),$ where $a_{k}: End(E_{k}\otimes L_{k})\cong End(E_{k})$ for $k=1, 2.$ Moreover, $((L_{1}\otimes E_{1}, \tau_1 a_1), \beta\otimes \nu, (L_{2}\otimes E_{2}, \tau_2 a_2))$ is isomorphic to $((E_{1}^{'}, \tau_{1}^{'}), \nu^{'}, (E_{2}^{'}, \tau_{2}^{'}))$ in $G_{A}(U).$ 
\end{enumerate}
\end{lemma}

\begin{proof}
For simplicity, we write $E(E_{k})$ instead of $End(E_{k})$ throughout the proof.

 (1)  By definition,  for $k=1, 2$, we have $$ \tau_{k}: E(E_{k})\cong j^{*}A_{k},  \tau_{k}^{'}: E(E_{k}^{'})\cong j^{*}A_{k},$$ where $j$ denotes the \'etale map $U \to S.$ Then $\tau_{k}^{'-1} \tau_{k}: E(E_{k})\cong E(E_{k}^{'}).$ So, we can find unique invertible sheaves  $L_{k}= E_{k}^{'}\otimes_{E(E_{k})} E_{k}^{\vee}$ such that (see Lemma 4.3 of \cite{VNR})
 $$v_{k}: L_{k}\otimes_{\mathcal{O}_{U}} E_{k}=(E_{k}^{'}\otimes_{E(E_{k})} E_{k}^{\vee})\otimes_{\mathcal{O}_{U}} E_{k}\cong  E_{k}^{'}\otimes_{E(E_{k})} (E_{k}^{\vee}\otimes_{\mathcal{O}_{U}} E_{k})\cong E_{k}^{'}.$$ Note that $v_{k}$'s are $( E(E_{k}), \mathcal{O}_{U})$-linear isomorphisms. By using $\nu, \nu^{'}$ and $v_{k},$ we obtain an $\mathcal{O}_{X \times_{S} U}$-module isomorphism 
 
 \begin{align*}
    f_{U}^{*}L_{1}\otimes  f_{U}^{*}L_{2}^{-1}\otimes  f_{U}^{*}E_{1}\otimes ( f_{U}^{*}E_{2})^{\vee}\cong  f_{U}^{*}E_{1}^{'}\otimes ( f_{U}^{*}E_{2}^{'})^{\vee} \cong ( f_{U}^{*}E_{1})^{' \vee}\otimes  f_{U}^{*}E_{2}^{'} \\\cong  f_{U}^{*}L_{1}^{-1}\otimes  f_{U}^{*}L_{2}\otimes ( f_{U}^{*}E_{1})^{\vee} \otimes  f_{U}^{*}E_{2}\cong  f_{U}^{*}L_{1}^{-1}\otimes  f_{U}^{*}L_{2}\otimes  f_{U}^{*}E_{1}\otimes ( f_{U}^{*}E_{2})^{\vee}.
 \end{align*}
 
 Set $$L= f_{U}^{*}L_{1}\otimes  f_{U}^{*}L_{2}^{-1},    L^{\vee}= f_{U}^{*}L_{1}^{-1}\otimes  f_{U}^{*}L_{2},$$
     $$E= f_{U}^{*}E_{1}\otimes ( f_{U}^{*}E_{2})^{\vee},  E^{\vee}= ( f_{U}^{*}E_{1})^{\vee} \otimes  f_{U}^{*}E_{2},$$
     $$E^{'}=  f_{U}^{*}E_{1}^{'}\otimes ( f_{U}^{*}E_{2}^{'})^{\vee}, (E^{'})^{\vee}= ( f_{U}^{*}E_{1}^{'})^{ \vee}\otimes  f_{U}^{*}E_{2}^{'}.$$ Thus, we have 
     \begin{align}\label{comm3}
      L\otimes E \stackrel{\Gamma}\to E^{'} \stackrel{\nu^{'}}\to (E^{'})^{\vee} \stackrel{\Gamma^{\vee}}\to L^{\vee}\otimes  E^{\vee} \stackrel{1_{L^{\vee}}\otimes \nu^{-1}} \to L^{\vee}\otimes  E, 
     \end{align} where $1_{L^{\vee}}$ denotes the identity map on $L^{\vee}.$
 
Now the Lemma \ref{iso lemma} implies that there is an $\mathcal{O}_{X\times_{S} U}$-module isomorphism $\beta: L \cong L^{\vee}.$ More precisely, we get $\beta$ as follows: 
\begin{equation}\label{map}
 L \stackrel{a}\to L\otimes E \otimes_{E(E)} E^{\vee} \stackrel{ \delta\otimes 1_{E^{\vee}}}\longrightarrow L^{\vee}\otimes  E^{\vee}\otimes_{E(E)} E^{\vee} \stackrel{1_{L^{\vee}}\otimes \nu^{-1}\otimes 1_{E^{\vee}}}\longrightarrow L^{\vee}\otimes  E \otimes_{E(E)} E^{\vee} \stackrel{b}\to L^{\vee}
\end{equation}

 Here $\delta= \Gamma^{\vee} \nu^{'} \Gamma$. 
  Therefore, $[L_{1}, \beta, L_{2}] \in \Tilde \Pic^{f}(U).$ This proves (1).

(2) Note that $\beta\otimes \nu$ gives an isomorphism between $L\otimes E$ and $L^{\vee}\otimes E^{\vee}.$  Hence, clearly $((L_{1}\otimes E_{1}, \tau_1 a_1), \beta\otimes \nu, (L_{2}\otimes E_{2}, \tau_2 a_2))$ defines an object in $G_{A}(U),$ where $a_{k}: E(E_{k}\otimes L_{k})\cong E(E_{k})$ for $k=1, 2.$

Next, we shall show that $(v_{1}, v_{2})$ defines an isomorphism between $((L_{1}\otimes E_{1}, \tau_1 a_1), \beta\otimes \nu, (L_{2}\otimes E_{2}, \tau_2 a_2))$ and $((E_{1}^{'}, \tau_{1}^{'}), \nu^{'}, (E_{2}^{'}, \tau_{2}^{'}))$ in $G_{A}(U).$ To prove $(v_{1}, v_{2})$ is an isomorphism in $G_{A}(U)$, we have to check that the diagrams like (\ref{commutativity}) and (\ref{comm2}) commute. To check these commutativity, we may assume that $U$ is affine.

We first check that the following diagram 
    $$ \xymatrix{
 E(E_{1}^{'}) \ar[rd]^{\tau_{1}^{'}} \ar[r]^{E(v_{1})} & E(L_{1}\otimes E_{1}) \ar[d]^{\tau_{1}a_{1}} \\
 &  j^{*}A_{1}} $$ commutes. 
 
 For $s\in E(E_{1}^{'}),$ we have $E(v_{1})(s)= v_{1}^{-1}sv_{1}.$ Note that $E(v_{1})^{-1}= E(v_{1}^{-1}).$ Let $\tau= \tau_{1}^{'-1} \tau_{1}.$ So, it suffices to check that $ E(v_{1}^{-1})a_{1}^{-1}= \tau$, i.e., $ (E(v_{1}^{-1})a_{1}^{-1})(t)= \tau(t)$ for all $t\in E(E_{1}).$ We can write $a_{1}^{-1}(t)= 1_{L_{1}}\otimes t.$ Moreover, $v_{1}$ is $E(E_{1})$-linear map and $E_{1}^{'}$ is $E(E_{1})$-module via $\tau.$ Using these facts, we get $((E(v_{1}^{-1})a_{1}^{-1})(t))(e^{'})=\tau(t)(e^{'})$ for all $e^{'}\in E_{1}^{'}.$ Hence the assertion.
 
 Further, we show that the diagram like (\ref{comm2}) commutes. By (\ref{comm3}), we have the following commutative diagram 
 $$\begin{CD}
   L\otimes E @>{\Gamma}>> E^{'} \\
   @VV\delta V        @VV \nu^{'}V \\
   L^{\vee}\otimes  E^{\vee} @> (\Gamma^{\vee})^{-1}>> (E^{'})^{\vee}.
 \end{CD}$$ We claim that $\beta\otimes \nu= \delta.$ Recall that $\beta= b( 1_{L^{\vee}}\otimes \nu^{-1}\otimes 1_{E^{\vee}})(\delta\otimes 1_{E^{\vee}})a.$ Then $\beta\otimes \nu=  (b\otimes \nu)( 1_{L^{\vee}}\otimes \nu^{-1}\otimes 1_{E^{\vee}} \otimes 1_{E}) (\delta \otimes 1_{E^{\vee}} \otimes 1_{E})( a\otimes 1_{E}).$ We get the following commutative diagram
$$\begin{CD}
   L\otimes E @>{\Gamma_{1}}>> (L\otimes E)\otimes_{E(E)} E(E) @> m >> L\otimes E @>{\Gamma}>> E^{'} \\
   @VV\beta\otimes \nu V     @VV \Gamma_{2} V             @VV\delta V   @VV \nu^{'}V \\
   L^{\vee}\otimes  E^{\vee} @> \Gamma_{3}^{-1} >>  (L^{\vee}\otimes  E^{\vee})\otimes_{E(E)} E(E) @> n >>  L^{\vee}\otimes  E^{\vee} @> (\Gamma^{\vee})^{-1}>> (E^{'})^{\vee}.
  \end{CD}$$
  Here $\Gamma_{1}= a\otimes 1_{E},$ $\Gamma_{2}= \delta \otimes 1_{E^{\vee}} \otimes 1_{E},$  $\Gamma_{3}= (b\otimes \nu).( 1_{L^{\vee}}\otimes \nu^{-1}\otimes 1_{E^{\vee}} \otimes 1_{E})$ and $m, n$ are natural isomorphisms sending $x\otimes \tilde e$ to $x \tilde e$ for all $\tilde e \in E(E).$ One can easily check that $m\Gamma_{1}$ and $n\Gamma_{3}^{-1} $ both are identity isomorphisms. Hence $\beta\otimes \nu= \delta.$ This completes the proof. \end{proof}

Given an \'etale map $U \to S,$  we define a group action $\rho(U): \Tilde \Pic^{f}(U) \times [G_{A}](U) \to [G_{A}](U)$  by 
 $$([L_{1}, \beta, L_{2}], [((E_{1}, \tau_{1}), \nu, (E_{2}, \tau_{2}))])\mapsto [((L_{1}\otimes E_{1}, a_1\tau_1), \beta\otimes \nu, (L_{2}\otimes E_{2}, a_2\tau_2)].$$
Here $L_{1}, L_{2}$ are line bundles on $U$ with an isomorphism  $\beta:  f_{U}^{*}L_{1}\otimes  f_{U}^{*}L_{2}^{-1}\cong  f_{U}^{*}L_{1}^{-1}\otimes  f_{U}^{*}L_{2}$ and $a_{k}: End(E_{k}\otimes L_{k})\cong End(E_{k})$ for $k=1,2.$

\begin{proposition}\label{pic-torsor}
Let $A:=(A_{1}, \alpha, A_{2})$ be in $\tilde Az(f^{*})$(or in $ Az(f^{*})$). Assume that $S$ is connected. Then $[G_{A}]$ is a $\Tilde \Pic^{f}$-torsor.
\end{proposition}
\begin{proof}
 We need to show that the following are true:
 \begin{enumerate}
 \item the action $\rho(U)$ is simply transitive,
 \item for every $U\in ob(S_{et})$ there exist a \'etale covering $\{U_{i} \to U\}$ such that $[G_{A}](U_{i})\neq \emptyset$ for all $i.$
\end{enumerate}

Let $[(E_{1}, \tau_{1}), \nu, (E_{2}, \tau_{2}) ],$ $[(E_{1}^{'}, \tau_{1}^{'}), \nu^{'}, (E_{2}^{'}, \tau_{2}^{'})]\in [G_{A}](U).$  Then by Lemma \ref{line element}, we can construct a unique $[L_{1}, \beta, L_{2}]\in \Tilde \Pic^{f}(U)$ such that $$\rho(U)([L_{1}, \beta, L_{2}], [(E_{1}, \tau_{1}), \nu, (E_{2}, \tau_{2}) ])= [(E_{1}^{'}, \tau_{1}^{'}), \nu^{'}, (E_{2}^{'}, \tau_{2}^{'})] .$$ This proves (1).

Since $S$ is connected, $A_{1}$ and $A_{2}$ have constant rank say $n^2$ and $k^2.$ Let  $j: U \to S$ be an \'etale map. Then $j^{*}A_{1}$ and $j^{*}A_{2}$ are Azumaya algebras on $U$ of rank $n^2$ and $k^2.$ By Proposition IV.2.1(d) of \cite{Milne}, there exist an \'etale covering $\{b_{i}:U_{i} \to U\}$ such that $$\tau_{1}: M_{n}(\mathcal{O}_{U_{i}})\cong b_{i}^{*}j^{*}A_{1}$$
                      $$\tau_{2}: M_{k}(\mathcal{O}_{U_{i}})\cong b_{i}^{*}j^{*}A_{2}$$ for each $i.$ This implies that $[(\mathcal{O}_{U_{i}}^{\oplus n}, \tau_{1}), \nu, (\mathcal{O}_{U_{i}}^{\oplus k}, \tau_{2})]\in [G_{A}](U_{i})$ for all $i.$ Here $\nu$ is the obvious isomorphism.  Hence  $[G_{A}](U_{i})\neq \emptyset$ for all $i.$ This proves (2). \end{proof}

    \begin{remark}\label{connectedness}{\rm
     Note that the action $\rho(U)$ is simply transitive without the assumption $S$ being connected.}
    \end{remark}

 \begin{lemma}\label{welldefineness}
      Let $\psi: A:=(A_{1}, \alpha, A_{2}) \to B:= (B_{1}, \beta, B_{2})$ be a morphism in $\Tilde Az(f^{*}).$ Then the induced map $[G_{A}] \to [G_{B}]$ is an isomorphism as $\Tilde \Pic^{f}$-torsors.
    \end{lemma}
    
    \begin{proof}
    Since $\alpha$ and $\beta$ are morphisms in $\tilde Az(X),$  $\alpha = [(P, u, Q)]$ and $\beta= [(R, v, T)].$ We have $u: f^{*}A_{1} \otimes End(P)\cong f^{*}A_{2} \otimes End(Q)$ and $v: f^{*}B_{1} \otimes End(R)\cong f^{*}B_{2} \otimes End(T).$  Write $\psi= (\psi_{1}, \psi_{2}),$ where $\psi_{1}:= [(P_{1}, u_{1}, Q_{1})]$ and $ \psi_{2}:= [(R_{1}, v_{1}, T_{1})]$ are morphisms in $\tilde Az(S).$ In view of Remark \ref{equi}, we prefer to write $\alpha = [(P,  Q)],$ $\beta= [(R, T)],$   $\psi_{1}= [(P_{1}, Q_{1})]$ and $ \psi_{2}= [(R_{1}, T_{1})].$ Since $\psi$ is a morphism, the following diagram
    $$\begin{CD}
       f^{*}A_{1} @> [(P,~ Q)]>> f^{*}A_{2}\\
       @V [(f^{*}P_{1},~ f^{*}Q_{1})]VV    @V [(f^{*}R_{1},~ f^{*}T_{1})]VV \\
       f^{*}B_{1} @> [(R,~ T)]>> f^{*}B_{2}
      \end{CD}$$ commutes, i.e., 

     $$[(f^{*}R_{1}\otimes P,  f^{*}T_{1}\otimes Q)]= [(R\otimes f^{*}P_{1},  T\otimes f^{*}Q_{1})].$$ This means that there exist self dual locally free $\mathcal{O}_{X}$-modules $H_{1}$ and $H_{2}$ of finite rank over $X$ such that 
     
     \begin{align}\label{equi2}
      f^{*}R_{1}\otimes P\otimes H_{1}\cong R\otimes f^{*}P_{1}\otimes H_{2},~  f^{*}T_{1}\otimes Q\otimes H_{1}\cong T\otimes f^{*}Q_{1}\otimes H_{2}. 
     \end{align}

     Let $\phi_{1}= ([(P_{1}, \mathcal{O}_{S})], [(R_{1}, \mathcal{O}_{S})])$ and $\phi_{2}= ([(Q_{1}, \mathcal{O}_{S})], [(T_{1}, \mathcal{O}_{S})]).$ 
     
     We claim that
     $$\phi_{1}: A:=(A_{1}, \alpha, A_{2}) \to A^{'}:=(A_{1}\otimes End(P_{1}), \alpha^{'}, A_{2}\otimes End(R_{1})),$$  $$\phi_{2}:B:=(B_{1}, \beta, B_{2}) \to B^{'}:=(B_{1}\otimes End(Q_{1}), \beta^{'}, B_{2}\otimes End(T_{1}))$$ both are morphisms in $\tilde Az(f^{*}),$ where 
     $ \alpha^{'}:=[(P\otimes f^{*}R_{1},  Q \otimes f^{*}P_{1})]~  {\rm and}~ \beta^{'}:= [(R\otimes f^{*}T_{1},  T\otimes f^{*}Q_{1})].$ Clearly, $[(P_{1}, \mathcal{O}_{S})]$ and $ [(R_{1}, \mathcal{O}_{S})])$ both are morphisms in $\tilde Az(S).$  Note that $(P\otimes f^{*}R_{1}\otimes f^{*}P_{1}, Q\otimes  f^{*}P_{1})\sim (P\otimes f^{*}R_{1}, Q)$ in $\tilde \Delta(f^{*}A_{1}, f^{*}A_{2} \otimes f^{*}End(R_{1}) ).$ Thus the following diagram $$\begin{CD}
       f^{*}A_{1} @> [(P,~ Q)]>> f^{*}A_{2}\\
       @V [(f^{*}P_{1},~ \mathcal{O}_{X})]VV    @V [(f^{*}R_{1},~ \mathcal{O}_{X})]VV \\
       f^{*}A_{1}\otimes f^{*}End(P_{1}) @> [(P\otimes f^{*}R_{1},  ~Q\otimes f^{*}P_{1})]>> f^{*}A_{2} \otimes f^{*}End(R_{1})
      \end{CD}$$ commutes. Therefore, $\phi_{1}$ is a morphism in $\tilde Az(f^{*}).$ Similarly for $\phi_{2}.$ Hence the cliam.

      Moreover $\alpha^{'} \sim \beta^{'},$ because there exist self dual $\mathcal{O}_{X}$-modules  $ f^{*}T_{1}\otimes H_{1}, f^{*}P_{1}\otimes H_{2}$ such that (by using (\ref{equi2}))
      $$ P\otimes f^{*}R_{1}\otimes f^{*}T_{1}\otimes H_{1}\cong R \otimes f^{*}T_{1}\otimes f^{*}P_{1}\otimes H_{2},~~ Q\otimes f^{*}P_{1}\otimes f^{*}P_{1}\otimes H_{2}\cong T\otimes f^{*}Q_{1}\otimes f^{*}P_{1}\otimes H_{2}.$$ This implies that $([(\mathcal{O}_{S}, \mathcal{O}_{S})], [(\mathcal{O}_{S}, \mathcal{O}_{S})])$ defines a morphism $A^{'} \to B^{'}$ in $\tilde Az(f^{*}).$
      
      Now for  an \'etale map $j:U \to S,$ $\phi_{1}$ induces a $\Tilde \Pic^{f}(U)$-linear map $[G_{A}](U) \to [G_{A^{'}}](U)$ by sending  
     $$[(E_{1}, \tau_{1}), \nu, (E_{2}, \tau_{2})]  ~ {\rm{to}}~ [(E_{1}\otimes j^{*}P_{1}, \tau_{1}^{'}), \nu^{'}, (E_{2}\otimes j^{*}R_{1}, \tau_{2}^{'})],$$  where $\tau_{1}^{'}: End(E_{1}\otimes j^{*}P_{1})\cong j^{*}(A_{1}\otimes  End(P_{1})),$ $ \tau_{2}^{'}: End(E_{2}\otimes j^{*}R_{1})\cong j^{*}(A_{2}\otimes  End(R_{1}))$ and $\nu^{'}: f_{U}^{*}(E_{1}\otimes j^{*}P_{1})\otimes (f_{U}^{*}(E_{2}\otimes j^{*}R_{1}))^{\vee} \cong (f_{U}^{*}(E_{1}\otimes j^{*}P_{1}))^{\vee}\otimes f_{U}^{*}(E_{2}\otimes j^{*}R_{1})$ (for $\nu^{'},$ we use $\nu$ and the fact that  $P_1^{\vee}\cong P_1$ and $R_{1}^{\vee}\cong R_1$). Note that $\Tilde \Pic^{f}(U)$ acts on both $[G_{A}](U),$ $[G_{A^{'}}](U)$ simply transitively (see Remark \ref{connectedness}). Thus, we get $[G_{A}](U)\cong[G_{A^{'}}](U)$ as sets. Therefore $[G_{A}]\cong [G_{A^{'}}]$ as $\Tilde \Pic^{f}$-torsors.  Similarly for $\phi_{2},$  i.e., $[G_{B}]\cong [G_{B^{'}}].$ Clearly, $[G_{A^{'}}]\cong [G_{B^{'}}]$ because  $([(\mathcal{O}_{S}, \mathcal{O}_{S})], [(\mathcal{O}_{S}, \mathcal{O}_{S})])$ defines a morphism $A^{'} \to B^{'}$ in $\tilde Az(f^{*}).$  Hence the result. \end{proof}
     
     \begin{remark}\label{dual imp}{\rm
      We do not know whether the Lemma \ref{welldefineness} holds or not whenever $\psi$ is a morphism in $Az(f^{*}).$}
     \end{remark}

\subsection*{Presheaf contracted products}
  Let $\mathcal{G}$ be a presheaf of groups. Let $\mathcal{F}_{1}$ and $\mathcal{F}_{2}$ be two presheaves of sets on a site $\mathcal{C}$. Suppose that $\mathcal{F}_{1}$(resp. $\mathcal{F}_{2}$) has a right (resp. left) $\mathcal{G}$ action. We can consider the left action of $\mathcal{G}$ on the product given for every object $X$ by:
  $$(\mathcal{G} \times \mathcal{F}_{1} \times \mathcal{F}_{2})(X) \to (\mathcal{F}_{1} \times \mathcal{F}_{2})(X), (g, x, y)\mapsto (xg^{-1}, gy).$$
  
  Then the presheaf contracted product of $\mathcal{F}_{1}$ and $\mathcal{F}_{2}$ to be the presheaf defined as the quotient by the action of $\mathcal{G}$: $$X \mapsto \mathcal{F}_{1}(X) \times \mathcal{F}_{2}(X)/\sim,$$ i.e., the quotient by the equivalence relation define by $(xg, y)\sim (x, gy)$ for all $g\in \mathcal{G}(X).$  It is denoted by $\mathcal{F}_{1}\Pi^{\mathcal{G}} \mathcal{F}_{2}.$ We write $[x, y]$ for the element of  $\mathcal{F}_{1}\Pi^{\mathcal{G}} \mathcal{F}_{2}(X),$ where $X\in ob(\mathcal{C}),$ $x\in \mathcal{F}_{1}(X)$ and $y\in \mathcal{F}_{2}(X).$ For every $X\in ob(\mathcal{C}),$ the left action of $\mathcal{G}$ on the contracted product is given by
  $$(\mathcal{G} \times \mathcal{F}_{1} \Pi^{\mathcal{G}} \mathcal{F}_{2})(X) \to (\mathcal{F}_{1} \Pi^{\mathcal{G}} \mathcal{F}_{2})(X), (g, [x, y])\mapsto [g.x, y].$$

  Assume that $\mathcal{G}$ is a presheaf of abelian groups. Then there is no distinction between left and right $\mathcal{G}$-torsors. Let ${\rm {\rm Tors}}(S, \mathcal{G})$ denotes the set of isomorphism classes of $\mathcal{G}$-torsors. It is well known that the set ${\rm Tors}(S, \mathcal{G})$ has an abelian group structure under the operation presheaf contracted product with identity $\mathcal{G}$ and the inverse of $\mathcal{F}$ is  $\mathcal{F}$ itself with the $\mathcal{G}$ action $(g, x)\mapsto g^{-1}.x.$

   \begin{lemma}\label{respect tensor and composition}
   \begin{enumerate}
    \item Let $A:= (A_{1}, \alpha, A_{2})$,  $A^{'}:= (A_{1}^{'}, \alpha^{'}, A_{2}^{'})$ be in $\tilde Az(f^{*}).$ Write $A\otimes A^{'}:= (A_{1}\otimes A_{1}^{'}, \alpha\otimes \alpha^{'}, A_{2}\otimes A_{2}^{'}).$ Then $[G_{A\otimes A^{'}}]\cong [G_{A}]\Pi^{\Tilde Pic^{f}} [G_{A^{'}}]$ as $\Tilde \Pic^{f}$-torsors.
    \item Let $A:=(A_{1}, \alpha, A_{2})$, $B:= (A_{2}, \beta, A_{3})$ be in $\tilde Az(f^{*}).$ Write $B\circ A:= (A_{1}, \beta\alpha, A_{3}).$ Then $[G_{B\circ A} ]\cong [G_{A}]\Pi^{\Tilde Pic^{f}} [G_{B}]$ as $\Tilde \Pic^{f}$-torsors.
   \end{enumerate}
\end{lemma}

\begin{proof}
 (1)  For an \'etale map $U \to S,$ we have a $\tilde \Pic^{f}(U)$-linear map $$[G_{A}](U) \times [G_{A^{'}}](U)/\sim ~ \to [G_{A\otimes A^{'}}](U), $$ sending $$[[(E_{1}, \tau_{1}), \nu, (E_{2}, \tau_{2})], [(E_{1}^{'}, \tau_{1}^{'}), \nu^{'}, (E_{2}^{'}, \tau_{2}^{'})]]~ {\rm {to}}~ [[(E_{1}\otimes E_{1}^{'}, \tau_{1}\otimes \tau_{1}^{'}),  \nu \otimes \nu^{'}, (E_{2}\otimes E_{2}^{'},\tau_{2}\otimes \tau_{2}^{'})]].$$ Since $\tilde \Pic^{f}(U)$ acts simply transitively on $[G_{A}](U) \times [G_{A^{'}}](U)/\sim$ and  $[G_{A\otimes A^{'}}](U),$ we get $[G_{A}](U) \times [G_{A^{'}}](U)/\sim ~ \cong [G_{A\otimes A^{'}}](U)$ as sets. Therefore $[G_{A\otimes A^{'}}]\cong [G_{A}]\Pi^{\tilde Pic^{f}} [G_{A^{'}}]$ as $\Tilde \Pic^{f}$-torsors.
 
 (2) Let $[(E_{1}, \tau_{1}), \nu, (E_{2}, \tau_{2})]\in [G_{A}](U)$ and  $[(E_{2}^{'}, \tau_{2}^{'}), \nu^{'}, (E_{3}, \tau_{3})] \in [G_{B}](U).$ Since $(E_{2}, \tau_{2})$ and $(E_{2}^{'}, \tau_{2}^{'}) \in F_{A_{2}}(U),$ there exists a unique line bundle $L$ such that $E_{2}^{'}\cong E_{2}\otimes L$ (see Lemma 4.3 of \cite{VNR}). The map $\nu\otimes \nu^{'}$ gives an  isomorphism
 $$ f_{U}^{*}(E_{1}\otimes L) \otimes ( f_{U}^{*}E_{3})^{\vee} \otimes End( f_{U}^{*}E_{2})\cong ( f_{U}^{*}(E_{1}\otimes L))^{\vee} \otimes  f_{U}^{*}E_{3} \otimes End( f_{U}^{*}E_{2})$$ of $\mathcal{O}_{X\times_{S} U}$-modules. Then by applying Lemma \ref{iso lemma}, we get an  $\mathcal{O}_{X\times_{S} U}$-module isomorphism $\chi: f_{U}^{*}(E_{1}\otimes L) \otimes ( f_{U}^{*}E_{3})^{\vee}\cong ( f_{U}^{*}(E_{1}\otimes L))^{\vee} \otimes  f_{U}^{*}E_{3}.$ This shows that $[(E_{1}\otimes L, \tau_{1}a_{1}), \chi, (E_{3}, \tau_{3})] \in [G_{B\circ A} ](U),$ where $a_{1}: End(E_{1}\otimes L)\cong End(E_{1}).$ Now, we have a $\tilde \Pic^{f}(U)$-linear map $$[G_{A}](U) \times [G_{B}](U)/\sim ~ \to [G_{B\circ A}](U), $$ sending $[[(E_{1}, \tau_{1}), \nu, (E_{2}, \tau_{2})], [(E_{1}^{'}, \tau_{1}^{'}), \nu^{'}, (E_{3}, \tau_{3})]]~ {\rm{to}}~ [[(E_{1}\otimes L, \tau_{1}a_{1}), \chi, (E_{3},\tau_{3})]].$ The rest of the arguments similar to (1). Hence the result. \end{proof}

By Lemma \ref{welldefineness}, we obtain a well-define map 
\begin{equation}\label{first step}
 \{\rm{The~ set~ of~ isomorphism~ classes~ of~ objects~ in~ \tilde { \it Az( f^{*}})}\} \to {\it  {\rm Tors}(S,\Tilde \Pic^{{\it f}})}, {\it [A] \mapsto [G_{A}]}.
\end{equation}

In fact, this induces a well-define group homomorphism 

$$\omega: 
\tilde \Br(f) \to {\rm Tors}(S, \Tilde \Pic^{f})$$ by Lemma \ref{respect tensor and composition}. Moreover, $\omega$ is a natural map (see below).
 
\begin{lemma}\label{naturality}
 The map $\omega$ is natural.
\end{lemma}
\begin{proof}
Given a commutative diagram 

$$\begin{CD}
     X^{'} @> f^{'}>> S^{'}\\
    @V g VV  @V h VV \\                 
    X @> f >> S
  \end{CD},$$ we want to show that the following diagram 
  $$\begin{CD}
    \tilde \Br(f) @>>> {\rm Tors}(S, \Tilde \Pic^{f})\\
    @V h^{*} VV  @V {\rm Tors}(h) VV \\                 
   \tilde \Br(f^{'}) @>>> {\rm Tors}(S^{'}, \Tilde \Pic^{f^{'}})
  \end{CD}$$ is commutative. Here $h^{*}([(A_{1}, \alpha, A_{2})])=[ (h^{*}A_{1}, h^{*}\alpha, h^{*}A_{2})]$ and ${\rm Tors}(h)$ is defined as follows. For an \'etale map $V \stackrel{i}\to S^{'}$, $(h^{-1}\mathcal{G})(V):= \varinjlim  \mathcal{G}(U),$ where $\mathcal{G}\in {\rm Tors}(S, \Tilde \Pic^{f})$ and the direct limit is over the commutative diagrams 
  $$\begin{CD}
     V @> k >> U\\
    @V i VV  @V j VV \\                 
    S^{'} @> h >> S
  \end{CD}$$ with $j$ \'etale.  Now, for such a diagram,  we obtain a natural group homomorphism $$  \Tilde \Pic^{f}(U) \to \Tilde \Pic^{f^{'}}(V),  [L_1, a, L_2] \mapsto  [k^{*}L_1, \varepsilon^{*}a, k^{*}L_2].$$ Here $\varepsilon$ is the unique map $X^{'} \times_{S^{'}} V \to X\times_S U$ compatible with the maps $g$, $h$ and $k.$ By the universal property of direct limits, there is a unique homomorphism $h^{-1}\Tilde \Pic^{f} \to \Tilde \Pic^{f^{'}}.$ Then ${\rm Tors}(h)(\mathcal{G})= h^{-1}\mathcal{G} ~\Pi^{h^{-1}\Tilde \Pic^{f}} ~\Tilde \Pic^{{\it f^{'}}}.$ 
  
  More explicitly, we have to show that there is an isomorphism $$h^{-1}[G_{(A_1,~ \alpha,~ A_2)}]~{ \Pi^{h^{-1}\Tilde \Pic^{f}} ~ \Tilde \Pic^{{\it f^{'}}}} \cong [G_{(h^{*}A_1,~ h^{*}\alpha,~ h^{*}A_2)}]$$ of $\Tilde \Pic^{f^{'}}$-torsors. Note that we have a $ \Tilde \Pic^{f}(U)$-linear map 
  $$[G_{(A_1,~ \alpha,~ A_2)}](U) \to [G_{(h^{*}A_1,~ h^{*}\alpha,~ h^{*}A_2)}](V),$$ sending  $[(E_{1}, \tau_{1}), \nu, (E_{2}, \tau_{2})]$~ {\rm to}~ $[(k^{*}E_{1}, k^{*}\tau_{1}), \varepsilon^{*}\nu, (k^{*}E_{2}, k^{*}\tau_{2})].$ Therefore, we deduce a $\Tilde \Pic^{f^{'}}(V)$-linear map 
  $${h^{-1}[G_{(A_1,~ \alpha,~ A_2)}](V) \times \Tilde \Pic^{f^{'}}(V)/\sim ~} \to [G_{(h^{*}A_1,~ h^{*}\alpha,~ h^{*}A_2)}](V)$$ by passing to the limit over such all such diagrams. Since $ \Tilde \Pic^{f^{'}}(V)$ acts simply transitively on both sides, we get the result. \end{proof}

  Now, by using the maps $\Upsilon$ (see (\ref{Br to tilde Br})) and $\Psi$ (see (\ref{tor map})), we define a natural group homomorphism $\Br(f) \to {\rm Tors}(S, \Pic^{{\it f}})$ as follows
  \begin{equation}\label{second step}
   \Br(f) \stackrel{\Upsilon}\to \tilde \Br(f) \stackrel{\omega}\to {\rm Tors}(S, \tilde \Pic^{f}) \stackrel{{\rm Tors}(S, \Psi)}\longrightarrow {\rm Tors}(S, \Pic^{{\it f}}),
  \end{equation}

where ${\rm Tors}(S, \Psi)$ is the natural induced map sending any $\tilde \Pic^{f}$-torsor $\mathcal{G}$ to $\mathcal{G}~\Pi^{\tilde \Pic^{f}} \Pic^{{\it f}}.$
  
   \begin{lemma}\label{torsor and h1}
   There is a natural group homomorphism $\zeta: {\rm Tors}(S, \mathcal{G}) \to \check H_{et}^{1}(S, \mathcal{G}),$ where $\mathcal{G}$ is a presheaf of abelian groups and $\check H_{et}^{1}(S, \mathcal{G})$ denotes the first \'etale \v{C}ech cohomology group associated to $\mathcal{G}.$
  \end{lemma}
  
  \begin{proof}
   For the map $\zeta,$ we refer to section 11 of \cite{Milne1} or Proposition 4.6 of \cite{Milne}. The only difference is that we are defining the map here for presheaf torsors.  However, we observe that the sheaf properties have not been used in \cite{Milne1} to define such a map.  Hence, one can define such map in a similar way for presheaf torsors as well.
   
   By definition $\check H_{et}^{1}(S, \mathcal{G})= \varinjlim_{\mathcal{U}} \check H_{et}^{1}(\mathcal{U}, \mathcal{G}),$ where the limit is taken over all covering $\mathcal{U}=\{U_{i}\to S\}.$  Write $U_{ij}$ for $U_{i}\times_S U_{j}.$  Let $\mathcal{F}_{1}, \mathcal{F}_{2} \in {\rm Tors}(S, \mathcal{G}).$  Then for some covering $\mathcal{U}=\{U_{i}\to S\},$ $\zeta(\mathcal{F}_{1})= (g_{ij})_{(i, j)\in I \times I}$ and $\zeta(\mathcal{F}_{2})= (h_{ij})_{(i, j)\in I \times I},$ where $g_{ij}, h_{ij} \in \mathcal{G}(U_{ij}).$ Note that $(g_{ij})_{(i, j)\in I \times I}$ and $(h_{ij})_{(i, j)\in I \times I}$ satisfy the following properties: $g_{i, j}. x_{i}|_{U_{ij}}= x_{j}|_{U_{ij}}$ and $h_{i, j}. y_{i}|_{U_{ij}}= y_{j}|_{U_{ij}},$ where  $x_{i}\in \mathcal{F}_{1}(U_{i}),$  $y_{i}\in \mathcal{F}_{2}(U_{i}).$ 
   Recall that $(\mathcal{F}_{1} \Pi^{\mathcal{G}} \mathcal{F}_{2})(U_{i})$ consists of elements of the form $[x_{i}, y_{i}]$ with $x_{i}\in \mathcal{F}_{1}(U_{i}),$  $y_{i}\in \mathcal{F}_{2}(U_{i}).$ We have 
   $$ g_{ij}h_{ij}.[x_{i}{|_{U_{ij}}}, y_{i}{|_{U_{ij}}} ]= [g_{ij}h_{ij}.x_{i}{|_{U_{ij}}}, y_{i}{|_{U_{ij}}}]= [g_{ij}.x_{i}{|_{U_{ij}}}, h_{ij}. y_{i}{|_{U_{ij}}}]= [x_{j}{|_{U_{ij}}}, y_{j}{|_{U_{ij}}}].$$ This implies that $(g_{ij}h_{ij})_{(i, j)\in I \times I}$ defines an element of $\check H_{et}^{1}(\mathcal{U}, \mathcal{G}).$ Therefore $\zeta(\mathcal{F}_{1} \Pi^{\mathcal{G}} \mathcal{F}_{2})= (g_{ij}h_{ij})_{(i, j)\in I \times I}= \zeta(\mathcal{F}_{1}) \zeta(\mathcal{F}_{2})$ and  hence $\zeta$ is a group homomorphism. \end{proof}
   
  \begin{remark}\rm{
    If $\mathcal{G}$ is a sheaf of abelian groups then the map $\zeta$ is an isomorphism (see Proposition 11.1 of \cite{Milne1}).}\end{remark}
  
\begin{lemma}\label{same as etale}
 Let $\mathcal{P}ic^{f}_{et}$ be the \'etale sheafification of the presheaf $\Pic^{{\it f}}$ on $S_{et}.$ Then $$(f_{*}\mathcal{O}_{X}^{\times}/\mathcal{O}_{S}^{\times})_{et}\cong \mathcal{P}ic^{f}_{et}.$$
\end{lemma}
\begin{proof}
 For each \'etale  $U \to S,$ there is a natural map (using (\ref{U-Pic}))$$\Gamma(\tilde X, \mathcal{O}_{\tilde X})^{\times}/\Gamma(U, \mathcal{O}_{U})^{\times} \to \Pic^{{\it f}}({\it U}) \to \mathcal{P}ic^{{\it f}}({\it U}).$$ Here $\tilde X$ stands for $X\times_S U$. Then by the universal property of sheafification, we get a unique map $(f_{*}\mathcal{O}_{X}^{\times}/\mathcal{O}_{S}^{\times})_{et} \to \mathcal{P}ic^{f}_{et}$ of sheaves. Since $\Pic$ vanishes for a strictly hensel local ring, $ ((f_{*}\mathcal{O}_{X}^{\times}/\mathcal{O}_{S}^{\times})_{et})_{\bar{s}}\cong (\mathcal{P}ic^{f}_{et})_{\bar{s}}$ for all geometric point $\bar{s}$ of $S$ by the sequence (\ref{U-Pic}). Therefore, $ (f_{*}\mathcal{O}_{X}^{\times}/\mathcal{O}_{S}^{\times})_{et}\cong \mathcal{P}ic^{f}_{et}$. \end{proof}

\begin{theorem}\label{natural map}
  Let $f: X\to S $ be a faithful affine map of noetherian schemes. 
Then there is a natural group homomorphism $\theta: \Br(f)\to \Br^{'}(f):= H^{1}_{et}(S, f_{*}\mathcal{O}_{X}^{\times}/\mathcal{O}_{S}^{\times}).$
 \end{theorem}
 
 \begin{proof}
 Since $S$ is noetherian, we can write $S= \sqcup_{i=1}^{n} S_{i},$ where $S_{i}$'s are connected components of $S.$ Then $X= \sqcup_{i=1}^{n} X\times_{S} S_{i}$ and $f= \sqcup_{i=1}^{n} f_{i},$ where $f_{i}:  X\times_{S} S_{i} \to S_{i}.$ For each $\eta_{i}: f_{i} \to f,$ we get an induced map $\gamma_{i}:\Br(f) \to \Br(f_{i})$.  Therefore, $\gamma = (\gamma_{i})$ defines a map  $\Br(f) \to \Br(f_{1}) \times \dots \times \Br(f_{n}).$ Moreover, the cohomology $H^{i}$ commute with the finite disjoint union. So, we may assume that $S$ is connected. Now, we note that the following two well-known facts. 
 \begin{enumerate}
  \item Given a presheaf $\mathcal{F}$ of abelian groups on $S_{et}$  and the sheafification $\mathcal{F}^{+},$ there is always a natural map  $$\check H_{et}^{i}(S, \mathcal{F})\to \check H_{et}^{i}(S, \mathcal{F}^{+}).$$ 
  \item For all abelian sheaf $\mathcal{F},$ the natural map $$\check H_{et}^{1}(S, \mathcal{F})\to  H_{et}^{1}(S, \mathcal{F})$$ is an isomorphism. 
 \end{enumerate}

 The above two facts together with Lemmas \ref{torsor and h1} and \ref{same as etale} allow us to define the desired map $\theta$ as the composition of the following maps (see also (\ref{second step}))
 \begin{equation*}
 \begin{split}
  \Br(f) \stackrel{\Upsilon}\to \tilde \Br(f) \stackrel{\omega}\to {\rm Tors}(S, \tilde \Pic^{f}) \stackrel{{\rm Tors}(S, \Psi)}\longrightarrow {\rm Tors}(S, \Pic^{{\it f}})\stackrel{\zeta}\rightarrow \check {\it H}_{et}^{1}({\it S}, \Pic^{{\it f}})\to \check { \it H}_{et}^{1}({\it S}, \mathcal{P}ic^{{\it f}})\\  \stackrel{\cong}\to \check H_{et}^{1}(S, f_{*}\mathcal{O}_{X}^{\times}/\mathcal{O}_{S}^{\times}) \stackrel{\cong}\to H^{1}_{et}(S, f_{*}\mathcal{O}_{X}^{\times}/\mathcal{O}_{S}^{\times}).
 \end{split}
\end{equation*}\end{proof}

\begin{remark}\rm{
 The map $\theta$ is not injective in general (see Example \ref{not iso in gen}).}\end{remark}

 \section{The Relative Brauer group of subintegral extensions}\label{subintegral case}
In this section, we study the relative Brauer group $\Br(f)$ in the case when $f$ is a subintegral extension. Some details pertaining to subintegral extensions can be found in \cite{swan}.

\begin{theorem}\label{vanishing}
Let $f:A\hookrightarrow B$ be a subintegral extension of noetherian $\mathbb{Q}$-algebras. Then the following are true:
\begin{enumerate}
 \item  the natural map $f^{*}: \Br(A) \to \Br(B)$ is an isomorphism. 
 
 \item  $\Br(f)=0$.
\end{enumerate}
 
\end{theorem}

\begin{proof}
 (1) Since $f$ is subintegral, $B= \cup_{\lambda} B_{\lambda}$ where each $B_{\lambda}$ can be obtained by finite succession of elementary subintegral extension (i.e. $A\subset A[b]$ such that $b^2, b^3 \in A$) (see \cite{swan}). In other words, $f= \cup_{\lambda} f_{\lambda}$ where $f_{\lambda}: A\hookrightarrow B_{\lambda}$ is a finite map. For each $\lambda,$ we have the following two exact sequences
 $$ 0 \to \ker (f_{\lambda}^{*}) \to \Br(A) \stackrel{f_{\lambda}^{*}}\to \Br(B_{\lambda}), ~~   \Br(A) \stackrel{f_{\lambda}^{*}}\to \Br(B_{\lambda})\to \coker (f_{\lambda}^{*}) \to 0.$$
 
 As $\Br$ commutes with filtered limit of rings, $\Br(B)\cong \varinjlim_{\lambda} \Br(B_{\lambda}).$ Thus we get $\ker (f^{*})\cong \varinjlim_{\lambda} \ker (f_{\lambda}^{*})$ and  $\coker (f^{*})\cong \varinjlim_{\lambda} \coker (f_{\lambda}^{*}).$ So it is enough to show that $ \ker (f_{\lambda}^{*})=0= \coker (f_{\lambda}^{*})$ for each $\lambda.$

 By Proposition 5.4(3) of \cite{VS1}, $H_{et}^{i}({\rm Spec}(A), \mathcal{O}_{A}^{\times})\cong  H_{et}^{i}({\rm Spec}(B_{\lambda}), \mathcal{O}_{B_{\lambda}}^{\times})$ for all $i>1$ because $f_{\lambda}$ is a finite subintegral extension of $\mathbb{Q}$-algebras. In particular, we get $\Br^{'}(A)\cong \Br^{'}(B_{\lambda})$. Then the torsion subgroups $\Br^{'}(A)_{tor}$ and $\Br^{'}(B_{\lambda})_{tor}$ are also isomorphic. Therefore, the assertion follows from the following commutative diagram
 $$\begin{CD}
  \Br(A)@> f_{\lambda}^{*}>> \Br(B_{\lambda}) \\
  @VVV   @VVV\\
  \Br^{'}(A)_{tor} @>\cong>> \Br^{'}(B_{\lambda})_{tor},
 \end{CD}$$
 where the vertical maps are isomorphisms by a theorem of Gabber (see \cite{Djong}).
 
 (2) Since $f$ is subintegral, the map $\Pic({\it A})\rightarrow \Pic({\it B})$ is surjective by Proposition 7 of \cite{is}. Hence the result by (1) and the sequence (\ref{Pic-Br}). 
 \end{proof}
 
 Given an extension $f: A\hookrightarrow B,$ write $^{^+}\!\!\!f$ for the induced map $ ^{^+}\!\!\!A \hookrightarrow B$ where  $^{^+}\!\!\!A$ is the subintegral closure of $A$ in $B.$ 
 
 \begin{corollary}\label{rel Br same upto seminor}
  Let $f:A\hookrightarrow B$ be an extension of noetherian $\mathbb{Q}$-algebras. Then $\Br(f)\cong \Br(^{^+}\!\!\!f).$
 \end{corollary}
 \begin{proof}
  By comparing $\Br$-$\Pic$ sequences (see (\ref{Pic-Br})) for $f$ and $^{^+}\!\!\!f,$ we get the result. \end{proof}

We say that a faithful affine map $f: X \to S$ is {\it subintegral} if $\mathcal{O}_{S}(U)\to f_{*}\mathcal{O}_{X}(U)$ is subintegral for all affine open subsets $U$ of $S$ (see Definition 5.1 of \cite{VS1}).
 
 \begin{remark}\label{not true for nonaffine}(Cf. Example 6.6 of \cite{SW1}) \rm{
  Proposition \ref{vanishing}(2) is not true for non-affine schemes. For example, consider $S= \mathbb{P}_{\mathbb{Q}}^{1}$ and $X= \rm{{\rm Spec}} (\mathcal{O}_{B})$ where $\mathcal{O}_{B}= \mathcal{O}_{S}\oplus \mathcal{O}_{S}(-2)$ with $\mathcal{O}_{S}(-2)$ being a square zero ideal. Then $f: X \to S$ is a subintegral map and  $H=H^1(\mathbb{P}^1,\mathcal{O}_{S}(-2))$ is nonzero.
We also have $\Pic({\it X})=\Pic({\it S)\oplus H}.$ Therefore, $\Pic({\it S}) \to \Pic({\it X})$ is not surjective. Hence $\Br(f) \neq 0$ by the sequence (\ref{Pic-Br}). }\end{remark}

\section{Relative Kummer's sequence}\label{kummer}

For a commutative ring $A,$ let $\mu_{n}(A):= \ker [A^{\times} \stackrel{n} \to A^{\times}]$ and for a ring extension $f: A \hookrightarrow B,$ let $\mu_{n}(f):= \ker [\Pic({\it f}) \stackrel{{\it n}} \to \Pic({\it f})].$ We can identify $\Pic({\it f})$ with $I(f),$ the multiplicative group of invertible $A$-submodules of $B$ by Lemma 1.2 of \cite{SW1}. Some details related to $I(f)$ can be found in section 2 of \cite{rs}.
 
 \begin{lemma}\label{Kummer sequence locally}
  Let $f: A \hookrightarrow B$ be a finite ring extension. Assume that $A$ is a strictly hensel local ring with residue field $k$ and characteristic of $k$ does not divide $n.$ Then the following sequences 
  $$  0 \to \mu_{n}(A) \to \mu_{n}(B) \to \mu_{n}(f)\to 0,$$  $$  0 \to \mu_{n}(f) \to \Pic({\it f}) \stackrel{{\it n}} \to \Pic({\it f}) \to 0$$ are exact. 
   \end{lemma}
   
   \begin{proof}
    Since $B$ is finite over $A,$ $B$ is finite product of strictly hensel local rings. We know that for a strictly hensel local ring $A,$ the sequence 
    $$ 0 \to \mu_{n}(A) \to A^{\times} \stackrel{n}\to A^{\times}\to 0$$ is exact.
    Now the result follows from the following commutative diagram
    $$ \begin{CD}
        @.  0 @. 0 \\
           @. @VVV  @VVV   \\
        @.  \mu_{n}(A) @>>> \mu_{n}(B) @>>> \mu_{n}(f) \\
            @.  @VVV    @VVV   @VVV \\
        0 @>>> A^{\times} @>>> B^{\times} @>>> \Pic({\it f}) @>>> 0 \\
             @. @V n VV   @V n VV    @V n VV  \\
        0 @>>> A^{\times} @>>>B^{\times} @>>> \Pic({\it f}) @>>> 0 \\
         @.  @VVV   @VVV    \\
        @. 0 @. 0,
       \end{CD}$$ where the left two columns and  the bottom two rows are short exact sequences (see (\ref{U-Pic})). \end{proof}
       
      Let $f: X \to S$ be a faithful affine map of schemes. Recall that  $\mathcal{P}ic^{f}_{et}$ is the \'etale sheafification of the presheaf  $U\mapsto \Pic({\it f_{U}})$ on $S_{et},$ where $ f_{U}: X\times_S U \to U.$ For notational convenience, we prefer to write $\mathcal{I}_{et}$ instead of $\mathcal{P}ic^{f}_{et}\cong (f_{*}\mathcal{O}_{X}^{\times}/\mathcal{O}_{S}^{\times})_{et}$ (see Lemma \ref{same as etale}). 
      
      We write $\mathcal{\mu}_{n, X}$ (or simply $\mathcal{\mu}_{n}$) for the kernel of $\mathcal{O}_{X, et}^{\times} \stackrel{n}\to \mathcal{O}_{X, et}^{\times}.$ Similarly, $\mathcal{\mu}_{n}^{f}$ denotes the kernel of $\mathcal{I}_{et} \stackrel{n}\to \mathcal{I}_{et}.$

\begin{proposition}\label{Kummer sequence}
 Let $f: X \to S$ be a faithful finite map of schemes. Assume that the characteristic of $k(s)$ does not divide $n$ for any $s\in S.$ Then the following sequences
 \begin{align*}
  0 \to  \mu_{n} \to f_{*}\mu_{n} \to \mu_{n}^{f} \to 0,
 \end{align*}
 
 \begin{align*}
   0 \to \mu_{n}^{f} \to \mathcal{I}_{et}\stackrel{n}\to \mathcal{I}_{et} \to 0 
 \end{align*}

 of \'etale sheaves on $S_{et}$ are exact.
\end{proposition}

\begin{proof}
 Clear from Lemma \ref{Kummer sequence locally}.
\end{proof}

From the Proposition \ref{Kummer sequence}, we obtain the following two long exact sequences
 \begin{align}\label{kummer les1}
  \dots \to  H_{et}^{i}(S, \mu_{n})\to  H_{et}^{i}(X, \mu_{n}) \to H_{et}^{i}(S, \mu_{n}^{f})\to  H_{et}^{i+1}(S, \mu_{n}) \to \dots 
 \end{align}
 
 \begin{align}\label{kummer les2}
  \dots \to H_{et}^{i}(S, \mu_{n}^{f}) \to  H_{et}^{i}(S, \mathcal{I}_{et}) \stackrel{n}\to  H_{et}^{i}(S, \mathcal{I}_{et})\to  H_{et}^{i+1}(S, \mu_{n}^{f})\to \dots
 \end{align}
For a group $G,$ we denote the subgroup of elements of order dividing $n$ in $G$ by $_nG.$
 
\begin{theorem}\label{cohomology of roots unity}
 Let $f: X \to S$ be a faithful finite map of schemes. Assume that the characteristic of $k(s)$ does not divide $n$ for any $s\in S.$ 
 Then
 \begin{enumerate}
  \item $H_{et}^{0}(S, \mu_{n}^{f})\cong {_n\Pic({\it f})}$ and the sequence 
  \begin{align}\label{kummer ses}
    0\to \Pic({\it f}) \otimes \mathbb{Z}/{\it n}\mathbb{Z} \to { \it H_{et}^{1}(S, \mu_{n}^{f}) \to  {{\it _n{H}}_{et}^{1}(S, \mathcal{I}_{et})}} \to 0
  \end{align}
  is exact.

  \item if $f: L \hookrightarrow K$ is a finite field extensions and characteristic of $L$ does not divide $n$ then the sequence
  $$ 0 \to K^{\times}/L^{\times}\otimes \mathbb{Z}/n\mathbb{Z} \to H_{et}^{1}({\rm Spec}(L), \mu_{n}^{f}) \to  {_n\Br(K|L)} \to 0$$ is exact.
 \end{enumerate}
\end{theorem}

\begin{proof}

 (1) By Lemma 5.4 of \cite{SW}, $H_{et}^{0}(S, \mathcal{I}_{et})\cong \Pic({\it f}).$ The long exact sequence (\ref{kummer les2}) implies the assertion.

 (2) Note that $\Pic({\it f)\cong K^{\times}/L^{\times}}$ by (\ref{U-Pic}) and   $H_{et}^{1}({\rm Spec}(L), f_{*}\mathcal{O}^{\times}_K/\mathcal{O}^{\times}_{L})\cong\Br(K|L)$ by Lemma  \ref{rel}.  Hence the result by the short exact sequence (\ref{kummer ses}). \end{proof}
 
    Let  $\mathcal{P}ic^{{\it f}}_{zar}$ be the Zariski sheaf associated to the presheaf $\Pic^{{\it f}}$ on $S_{zar},$ defined as $\Pic^{{\it f}} ({\it U}) = \Pic({\it f_{U}}).$ A proof similar to Lemma \ref{same as etale} shows that $\mathcal{P}ic^{f}_{zar}\cong (f_{*}\mathcal{O}_{X}^{\times}/\mathcal{O}_{S}^{\times})_{zar}.$ We write $\mathcal{I}_{zar}$ for the Zariski sheaf $\mathcal{P}ic^{f}_{zar}\cong (f_{*}\mathcal{O}_{X}^{\times}/\mathcal{O}_{S}^{\times})_{zar}.$

\begin{theorem}\label{cohomology for subintegral map}
 Let $f: X \to S$ be a faithful finite map of noetherian schemes over $\mathbb{Q}.$
 \begin{enumerate}
  \item  If $f$ is a subintegral extension of affine schemes then $H_{et}^{i}(S, \mu_{n}^{f})=0$ for $i\geq 0.$ Moreover, $H_{et}^{i}(S, \mu_{n})\cong H_{et}^{i}(X, \mu_{n})$ for $i\geq 0.$
  \item If $f$ is subintegral and $S$ is a projective $\mathbb{Q}$-scheme then $H_{et}^{1}(S, \mu_{n}^{f})\cong  {_nH_{et}^{1}(S, \mathcal{I}_{et}}).$
  
 \end{enumerate}
\end{theorem}

\begin{proof}
 (1) For a subintegral extension $f: A\hookrightarrow B$ of $\mathbb{Q}$-algebras, we have $\Pic({\it f) \cong B/A}$ by Theorem 5.6 of \cite{rs} and Theorem 2.3 of \cite{rrs}. This implies that $\Pic({\it f})$ is a $\mathbb{Q}$-vector space, hence it is a divisible group. Therefore, the map $\Pic({\it f}) \stackrel{{\it n}} \to \Pic({\it f})$ is bijective. For $i>0,$  $H_{et}^{i}(S, \mathcal{I}_{et})\cong H_{zar}^{i}(S, \mathcal{I}_{zar})=0$ by Proposition 5.4 of \cite{VS1}. Hence, $H_{et}^{i}(S, \mu_{n}^{f})=0$ for $i\geq 0$ by (\ref{kummer les2}).
 
 The other part is clear from (\ref{kummer les1}).
 
 (2) By Proposition 5.3 of \cite{VS1}, $\Pic({\it f})\cong {\it H}_{zar}^{0}({\it S}, \it{f_{*}\mathcal{O}_{X}/\mathcal{O}_{S}})$. Note that $\it{f_{*}\mathcal{O}_{X}/\mathcal{O}_{S}}$ is coherent. So, $\Pic({\it f})$ is a $\mathbb{Q}$-vector space. Hence the result by the short exact sequence (\ref{kummer ses}).  \end{proof}

\end{document}